\documentclass{ifacconf}

\usepackage{graphicx}      
\usepackage{subcaption}   
\usepackage{amssymb}
\usepackage{amsmath,amsfonts,mathtools}
\usepackage{color}

\usepackage{tikz}
\usetikzlibrary{positioning,calc}
\definecolor{lightsteelblue2}{rgb}{0.737,0.824,0.933}

\usetikzlibrary{shapes, arrows.meta, positioning}

\newtheorem{assumption}{Assumption}
\newtheorem{objective}{Objective}
\newtheorem{remark}{Remark}
\newenvironment{proof}{\par\noindent Proof:\ }{\hfill$\blacksquare$\par}

\makeatletter
\let\old@ssect\@ssect 
\def\NR@nopatch@sectioning{}
\makeatother

\usepackage{natbib}
\usepackage{hyperref}
\makeatletter
\def\@ssect#1#2#3#4#5#6{%
   \NR@gettitle{#6}
   \old@ssect{#1}{#2}{#3}{#4}{#5}{\Sectionformat{#6}{#1}}%
}
\makeatother

\usepackage{fancyhdr}
\fancypagestyle{firstpage}{
    \fancyhf{}
    
    \fancyfoot[C]{\small © 2026 the authors. This work has been accepted to IFAC for publication under a Creative Commons Licence CC-BY-NC-ND}
}

\begin{document}
\begin{frontmatter}

\title{Towards Optimal Passive Feedback Control of LTI Systems under LQR Performance}

\author[First]{Armin Gießler} 
\author[First]{Pol Jané-Soneira} 
\author[First]{Sören Hohmann}

\address[First]{Institute of Control Systems,  Karlsruhe Institute of Technology, 76131 Karlsruhe, Germany (e-mail: \{armin.giessler,pol.soneira, soeren.hohmann\}@kit.edu).}

\begin{abstract}
   We study state-feedback design for continuous-time LTI systems with a control input and an external input-output pair. Our objective is to determine feedback gains that render the closed-loop system (strictly) passive with respect to the external port while minimizing the standard LQR cost in the disturbance-free case. The resulting constrained optimization problem is intractable due to bilinear matrix inequalities. We analyze the set of passivating gains, showing it is unbounded, possibly nonconvex, path-connected, and contractible. We propose an indirect approach, in which the set of passivating feedback gains is inner-approximated by a compact, convex polytope. A projected gradient flow is employed to compute a gain within this polytope that minimizes the LQR cost. Numerical examples illustrate the effectiveness of the method.
\end{abstract}

\begin{keyword}
Linear quadratic regulator, optimal control, passivity
\end{keyword}

\end{frontmatter}

\thispagestyle{firstpage}

\section{Introduction}

Dissipativity theory offers a powerful way to describe input-output behavior of dynamical systems through energy-based inequalities. Originating from the seminal work of \cite{willems1972dissipativeI,willems1972dissipativeII}, it now serves as a unifying framework for both input-output analysis and feedback design.
Passivity arises as a special case, corresponding to the supply rate that captures the exchange of power between inputs and outputs (\cite{brogliato_dissipative_2020}). 
For controllable and observable LTI systems, (strict) passivity is equivalent  to  (strict)  positive realness of its transfer function, as shown in \cite{Madeira2016}. 
While passive systems inherit strong stability properties, both for the isolated dynamics and their feedback interconnection, they may nevertheless exhibit poor transient performance. In this paper, we address this limitation by improving transient performance of (strictly) passive systems through a state-feedback controller and LQR performance criterion.

\subsection{Literature Review}
This literature review focuses on feedback controller design that ensures the closed-loop system is  passive or, equivalently, positive real (PR).
Fig.\ref{fig:structure1} illustrates two feedback configurations prevalent in the literature. 

In the top panel of Fig.\ref{fig:structure1}, the objective is to design a controller $H(s)$  such that the closed-loop is either stable ($W(s)=0$) or passive with respect to $(w,y)$.   
In \cite{LozanoLeal1988}, a strict PR (SPR) output controller $H(s)$ is sought for the PR system $G(s)$ to ensure closed-loop stability. The controller comprises a state observer and a state-feedback controller. 
The output-feedback design from \cite{Su2020} achieves input-feedforward or output-feedback passivity for SISO systems using a predefined  controller structure.
Robust passivation via parameter-dependent static output feedback is addressed in \cite{peaucelle2005robust}, ensuring that the closed-loop system remains passive under plant parameter variations.

\begin{figure}[t]
    \centering
    \begin{subfigure}[t]{0.355\columnwidth}
        \includegraphics[scale=0.73]{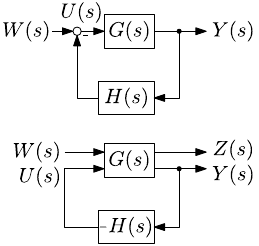}
        \caption{Systems in literature}
        \label{fig:structure1} 
    \end{subfigure}%
    \begin{subfigure}[t]{0.64\columnwidth}
        ~\includegraphics[scale=0.8]{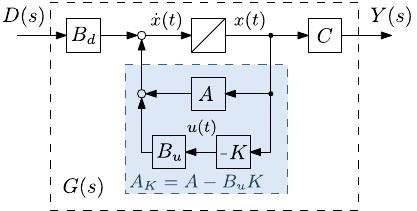}
        \caption{System under consideration}
        \label{fig:structure2} 
    \end{subfigure}
    \vspace{-0.3cm}
    \caption{Block diagrams for different feedback systems}
    \label{fig:structure} 
\end{figure}
In the lower panel of Fig.\ref{fig:structure1}, the PR system utilizes a performance input-output pair $(w,z)$ distinct from the control pair  $(u,y)$. The SPR controller of \cite{Haddad1994} extends the synthesis of \cite{LozanoLeal1988} to include an $H_\infty$-norm constraint on closed-loop performance.
Necessary and sufficient conditions for dynamic output-feedback and, if the state is measured, static state-feedback controllers that render the closed loop extended SPR are given in \cite{Sun1994}.
The problem of minimizing an upper bound on the closed-loop $H_2$-norm while constraining the controller to be SPR is studied in \cite{Geromel1997}, with \cite{shimomura2002strictly} reducing the conservatism via non-common Lyapunov functions.

In this work, the feedback structure under consideration is shown in Fig.~\ref{fig:structure2}. We seek a state-feedback controller that renders the closed-loop system   passive or strictly passive with respect to $(d,y)$, while minimizing the LQR cost of the system $(A,B_u)$. 
To the best of the authors' knowledge, the only existing approach that combines passivity with LQR performance is from \cite{Hallinan}. They show that for a specific class of feedback controllers synthesized from a quadratic storage function, optimality with respect to an LQR-like objective can be established. However, the existence of such a controller is not guaranteed, and the state-penalty matrix  
cannot be freely specified by the user, which thus severely limits the methodology.

\subsection{Contributions}
The contributions are threefold:

\begin{enumerate}
   \item We rigorously analyze the set of state-feedback gains that (strictly) passivate a continuous-time LTI system with respect to an external port, showing that the set is unbounded, possibly nonconvex, path-connected, and contractible.
\item The inherent intractability of LQR optimization under passivity constraints is identified, due to the presence of bilinear matrix inequalities.
   \item We propose an indirect approach that inner-\!\! approximates the set of passivating gains by a compact convex polytope and employs a projected gradient flow to compute a feedback gain that minimizes the LQR cost while guaranteeing (strict) passivity.
\end{enumerate}

Compared with the corresponding conference article, this extended version contains all proofs and additional details, corrects the condition for strict passivity, which requires $D+D^\top\succ 0$ for $D\neq 0$ (see Lemma~\ref{lem:KYP}), as well as minor errors, and provides additional simulation results for an open-loop unstable system and a feedback interconnection in Section~\ref{sec:simu}.

\subsection{Notation}
The set of real numbers is denoted by $\mathbb{R}$. 
For a matrix $A$, $A^\top, \operatorname{rank}(A), \operatorname{tr}\left(A  \right), \operatorname{vec}\left(A  \right)$, and $\det (A)$ denote its transpose, rank, trace, vectorization, and determinant, respectively. 
The identity matrix of dimensions $n\times n$ is given by $I_n$. A positive definite (semidefinite) matrix $A$ is denoted by $A\succ 0$ $ (A\succeq 0)$.
The interior, boundary, and closure of a set $\mathcal{S}$ are denoted by $\operatorname{int}(\mathcal{S}), \partial\mathcal{S}$, and $\operatorname{cl}(\mathcal{S})$, respectively.
The operator $\operatorname{diag}(\cdot)$ constructs a diagonal matrix from a vector. The gradient of a scalar function with respect to a vector $x$ (matrix $A$) is denoted by $\nabla f(x)$ $(\nabla f(A))$, and the Jacobian of a vector function $g$ is given by $J_g$.

\section{Preliminaries}
\subsection{Passivity}
Consider the LTI system 
\begin{subequations}
   \label{math:sys} 
   \begin{align}
 \dot{x}(t) & = A x(t) + B v(t), \label{math:lin_sys} \\
 z(t) & = C x(t) + D v(t),
\end{align} 
\end{subequations}
where $x(t)\in\mathbb{R}^n, v(t)\in\mathbb{R}^d$, and $z(t)\in\mathbb{R}^d$ denote the state, input, and output, respectively.
The system \eqref{math:sys} is strictly (state) passive if there exists a continuously differentiable positive semidefinite storage function $S(x)$ and a positive definite function $\psi(x)$ such that 
\begin{align}
   \label{math:diff_ineq} 
 \dot{S}(x(t)) \leq v^\top \!(t)  z(t) - \psi(x(t)).
\end{align} 
Setting $\psi(x) \equiv 0$ recovers the standard (nonstrict) passivity condition \cite[Def.~6.3]{khalil_nonlinear_2002}.

To treat the strict and nonstrict passivity conditions, as well as the cases $D = 0$ and $D \neq 0$, in a unified manner, we introduce the following notation. Define 
\begin{align}
 M &\coloneq \begin{bmatrix}
   A^\top P + P A& PB-C^\top \\ B^\top P - C & - D - D^\top       
 \end{bmatrix}, \label{math:M}  \\
 X \preceq_s 0 \; &\triangleq \;
\begin{cases}
X \prec 0, & \text{for strict passivity} \\
X \preceq 0, & \text{for nonstrict passivity}
\end{cases}, \label{math:X}
\end{align} 
\begin{align}
 L(D)  \coloneq \begin{cases}
   A^\top P + PA \preceq_s 0, \, B^\top P = C ,  & \text{for } D =0 \\
   M \preceq_s 0, & \text{for }  D\neq 0
 \end{cases}, \label{math:L} 
\end{align} 
and $L_{\text{s}}(D)$ and $ L_{ \text{ns} }(D)$ denote the strict and nonstrict versions of $L(D)$, respectively.

\begin{lem}
   \label{lem:KYP} 
Let the system \eqref{math:sys} be controllable and observable, and, for strict passivity with $D\neq 0$, let $D+D^\top\succ 0$. Then, the system \eqref{math:sys} is (strictly) passive if and only if there exists a matrix $P\succ 0$ such that the condition $L(D)$ holds. The associated storage function is $S(x) = \frac{1}{2} x^\top P x$.
\end{lem} 
Lemma~\ref{lem:KYP} follows from the conditions for (strict) positive realness\footnote{In particular, the conditions are the matrix inequalities associated with the state-space realization appearing in the Kalman-Yakubovich-Popov (KYP) Lemma.}  given in \cite[Sec.~4.6.3]{caverly2024lmi} and from the equivalence between (strict) positive realness and (strict) passivity shown in \cite{Madeira2016}.  

For $D=0$ or $D+D^\top\succ 0$, the condition of Lemma~\ref{lem:KYP} is equivalent to the existence of a matrix  $P\succ 0$ and a scalar  $\mu>0$ ($\mu=0$) such that
\begin{align}
  \begin{bmatrix}
   A^\top P + P A + \mu P & PB-C^\top \\ B^\top P - C & - D - D^\top       
 \end{bmatrix} \preceq 0,
\end{align} 
corresponding to strict (nonstrict) passivity.\footnote{
For $D=0$, the zero lower-right block implies $B^\top P = C$, and for $D+D^\top\succ 0$, the equivalence follows from the Schur complement. }
For $D\neq 0$, the strict condition $M\prec 0$ requires $D+D^\top\succ 0$ by the Schur complement and corresponds to extended SPR systems \cite[Lemma~2.3]{Sun1994}, which explains the restriction in Lemma~\ref{lem:KYP}. Hence, the results based on $L_{\text{s}}(D)$ also require $D+D^\top\succ 0$ for $D\neq 0$, whereas strictly passive systems with $D\neq 0$ and singular $D+D^\top$ are characterized by the above condition with $\mu>0$ \cite[Thm.~5]{Madeira2016}.


\subsection{Linear Quadratic Regulator}
The LQR objective is given by
\begin{align}
 f(x,u) = \int_0^\infty x(t)^\top Q x(t) + u(t)^\top R u(t) \mathrm{d}t, 
\end{align} 
where $u(t)\in\mathbb{R}^m$, $Q\succeq 0$, and $R\succ 0$.
The LQR problem is 
\begin{align}
 \min_{x,u} f(x,u) \quad \text{s.t.} \quad \dot{x} = Ax + B_u u  ~\; \text{and}  ~\; x(0) = x_0. \label{math:LQR_opt} 
\end{align} 
\begin{assumption}
   \label{ass:2} 
   The system $(A,B_u)$ is stabilizable and the pair $(A,\sqrt{Q})$ is detectable. 
\end{assumption}
Under Assumption~\ref{ass:2}, the optimal solution to \eqref{math:LQR_opt} is a stabilizing,
time-invariant state-feedback controller $u=-K^*x $ with $K^* = R^{-1}B^\top_u X^*$, where $X^*$ is the unique  positive semidefinite solution of the CARE 
\begin{align}
 A^\top X + X A - XB_u R^{-1}B^\top_u X + Q = 0.
\end{align} 
As $K^*$ does not depend on $x_0$, it also solves the parametrization of the LQR problem \eqref{math:LQR_opt} with respect to the feedback gain $K$ proposed by \cite{Feron}, with $x_0x_0^\top$ replaced by $I_n$,
\begin{subequations}
    \label{math:LQR_feron}
\begin{align}
    \min_{K,Y}~ &   \operatorname{tr }(QY) + \operatorname{tr }(K^\top R K Y)\\
    \text{s.t.}~ & (A-B_u K) Y + Y (A-B_u K)^\top + I_n \preceq 0, \label{math:LQR_feron_stab}\\
    & Y\succeq 0,
\end{align}
\end{subequations}
which can be convexified by a change of variables.
Alternatively, the  LQR problem \eqref{math:LQR_opt} can be formulated as \cite[p.~115]{LMIbook}
\begin{subequations}
\label{math:LQR_boyd}
\begin{align}
    \max_X ~& \operatorname{tr }(X) \\
    \text{s.t.}~& \begin{bmatrix} A^\top X + X A + Q & XB_u \\ B_u^\top X & R \end{bmatrix} \succeq 0, \label{math:LQR_boyd_care}\\
    & X\succeq 0.
\end{align}
\end{subequations}
The problems \eqref{math:LQR_feron} and \eqref{math:LQR_boyd} are duals as shown in \cite{Balakrishnan1995}. Problem \eqref{math:LQR_boyd} is a semidefinite program (SDP), and \eqref{math:LQR_feron} becomes one after the change of variables, so that both can be solved efficiently.

\subsection{Policy Gradient Flow for LQR}
We briefly summarize the policy gradient flow from \cite{bu2020clqr}.
The set of stabilizing feedback gains of the system $(A,B_u)$ is denoted by
\begin{align}
   \mathcal{K} & = \{ K \in \mathbb{R}^{m \times n} \mid \operatorname{Re}(\lambda_i(A-B_uK)) < 0 ~ \forall i \}, \label{math:K}
\end{align}
and is open, unbounded, and path-connected, as shown in \cite[Sec.~3]{bu2019}.
   The LQR cost function parametrized in terms of $K$ is defined as the matrix function $f_K:\mathbb{R}^{m\times n}\to \mathbb{R}$\footnote{Without loss of generality, we replaced $x_0x_0^\top$ by $I_n$ in the cost function $f_K = x_0^\top X_K x_0 = \operatorname{tr}\left(X_K x_0 x_0^\top \right)$ to obtain an initial-state independent formulation.}
    \begin{align}
     K\mapsto f_K= \operatorname{tr }(X_K), \label{math:cost_K} 
    \end{align} 
    where
    \begin{align}
       X_K & = \int_0^\infty e^{(A-B_uK)^\top t }\left(Q + K^\top R K  \right) e^{(A-B_uK) t}\mathrm{d}t. \label{math:int_P} 
    \end{align} 
 The function $f_K$ exhibits favorable analytic properties for its gradient flow. Its effective domain has interior $\mathcal{K}$, over which $f_K$ is real analytic, coercive, and gradient dominated on each of its sublevel sets. 
The gradient flow of \eqref{math:cost_K} is 
 \begin{align}
   \dot K = -  \nabla f_K,  \quad K(0)  = K_0\in\mathcal{K},  \label{math:grad}
\end{align}
where 
\begin{align}
   \nabla f_K & = 2 \left( R K - B_u^\top X_K \right) Y_K, \label{math:grad_mod}  \\
    Y_K & =  \int_0^\infty e^{(A-B_uK) t } I_n e^{(A-B_uK)^\top t} \mathrm{d}t. \label{math:Y} 
\end{align}
For each $K\in\mathcal{K}$, $X_K$ and $Y_K$ are the unique, positive semidefinite solutions to
the Lyapunov equations 
\begin{align}
    0 & = (A-B_uK)^\top \! X_K \!+\! X_K (A-B_uK) \!+\! Q \!+\! K^\top \! R K, \label{math:Lyap1} \\
    0 & = (A-B_uK) Y_K + Y_K (A-B_uK)^\top + I_n, \label{math:Lyap2}
\end{align}
respectively.
The gradient flow \eqref{math:grad} is well-defined for $K\in\mathcal{K}$ and generates unique trajectories $K(t)$ within $\mathcal{K}$ that converge to the optimal feedback $K^*$. 

\subsection{Projected Gradient Flow}
\label{subsec:proj} 
We provide a brief summary of the projected gradient flow from \cite{jongen2003constrained}.
Consider the inequality-constrained, possibly nonconvex optimization problem 
\begin{align}
   \min_{x\in\mathbb{R}^n}  ~f(x) \quad  \text{s.t.} \quad  g(x) \geq  0, \label{math:opt1}
\end{align}
where $f:\mathbb{R}^n\to \mathbb{R}$ and $g:\mathbb{R}^n\to \mathbb{R}^l$ are smooth functions, i.e., $f,g \in C^\infty(\mathbb{R}^n)$. 
The feasible region is defined by the  set $\mathcal{M}=\{x\in\mathbb{R}^n \mid g(x) \geq 0 \}$ which is assumed to be compact and connected. Additionally, it is assumed that the Linear Independence Constraint Qualification (LICQ) is satisfied at all $x\in \mathcal{M}$. 
The projection of $-\nabla f(x)$ onto the tangent cone  $C_x \mathcal{M}$ of $\mathcal{M}$ at $x\in\mathcal{M}$ is the
unique solution to
\begin{align}
   \min_\xi \Vert \xi + \nabla f(x) \Vert_2 \quad \text{s.t.} \quad \xi \in C_x \mathcal{M}. \label{math:proj}
\end{align}
Since the resulting vector field may be discontinuous at $\partial\mathcal{M}$, \cite{jongen2003constrained} propose the smooth alternative $\Pi_\mathcal{M} \nabla f(x)=M(x)\nabla f(x)$, where
   \begin{align}
      M(x) &= I_n - J_g^\top(x) F(x)^{-1} J_g(x), \label{math:projected_grad} \\
      F(x) &= 2 \operatorname{diag }\left(g\left(x\right)\right) + J_g(x) J_g^\top(x). \label{math:f_x}
   \end{align}
The dynamics of the projected gradient flow is given by 
\begin{align}
   \label{math:proj_grad} 
   \dot x = - \Pi_\mathcal{M} \nabla f(x)= - M(x) \nabla f(x), \quad  x(0)\in \mathcal{M}.  
\end{align}
The following proposition is taken from \cite[Prop.~3.2]{jongen2003constrained}.
\begin{prop}
   \label{prop:invariant}
   The vector field $x\mapsto -\Pi_\mathcal{M} \nabla f(x)$ is $C^\infty$ smooth on $\mathcal{M}$, and it induces a smooth trajectory $x(t)$ with $\operatorname{int }(\mathcal{M})$ and $\partial \mathcal{M}$ as invariant manifolds.
\end{prop}

Hence, trajectories starting in $\operatorname{int}(\mathcal{M})$ remain in $\operatorname{int}(\mathcal{M})$. If all critical points of \eqref{math:opt1} are nondegenerate, every trajectory converges to a critical point, and the asymptotically stable ones are exactly the local minimizers \cite[Thms.~4.2 and~4.4]{shikhman2009constrained}, so that \eqref{math:proj_grad} locally solves \eqref{math:opt1}.

\section{Problem Formulation}
The LTI systems in this work obey the dynamics 
\begin{subequations}
      \label{math:lin_sys_} 
\begin{align}
 \dot x(t) & = A x(t) + B_u u(t) + B_d d(t), \label{math:lin_sys_1}  \\
 y(t) & = C x(t) + D d(t), 
\end{align} 
\end{subequations}
where $u(t)\in\mathbb{R}^{m}$ is the control input, $ d(t)\in\mathbb{R}^p$ is the external input or disturbance, and $ y(t)\in\mathbb{R}^p$ is the output.
The closed-loop system under state-feedback $u=-Kx$ is 
\begin{subequations}
      \label{math:closed_lin} 
\begin{align}
 \dot x(t) & = A_K x(t) + B_d d(t), \\
 y(t) & = C x(t)+D d(t), 
\end{align} 
\end{subequations}
where $A_K \coloneq A-B_uK$.
In contrast to \eqref{math:sys}, the system \eqref{math:closed_lin} has input-output pair $(d,y)$ instead of $(v,z)$.

The following assumption holds throughout the paper. 
\begin{assumption}
   \label{ass:c+o} 
   The system  $(A,B_u)$ is stabilizable and the pair $(A,\sqrt{Q})$ is detectable. For each admissible feedback gain $K$, the system $(A_K,B_d)$ is controllable and the pair $(A_K,C)$ is observable.
\end{assumption}
Controllability of $(A_K,B_d)$ and observability of $(A_K,C)$ establish that any  quadratic storage function must be positive definite, i.e., $S(x)>0$ for $x\neq 0$, and preclude hidden, unstable dynamics that could violate the passivity inequality.\footnote{This ensures that the passivity condition \eqref{math:diff_ineq} holds uniformly across all possible states of the system \eqref{math:closed_lin}.}
The stabilizability of  $(A,B_u)$ and detectability of $(A,\sqrt{Q})$ ensure that the optimal LQR controller $u=-K^*x$ stabilizes the system \eqref{math:lin_sys_1} with $d=0$.

\begin{objective}
   \label{obj} 
   Determine a state-feedback gain $K$ such that the resulting closed-loop system \eqref{math:closed_lin} is (strictly) passive with respect to the input-output pair $(d,y)$, while minimizing the associated LQR cost in the disturbance-free case ($d=0$). Formally,
   \begin{subequations}
      \label{math:obj_prob}
   \begin{align} 
 \min_K~& \operatorname{tr}(X_K) \\
 \text{s.t.}~ & \eqref{math:closed_lin}\text{ being (strictly) passive},
\end{align}
\end{subequations}
where $X_K$ is given by \eqref{math:int_P}.
\end{objective}

For the closed-loop system \eqref{math:closed_lin}, we define $M_K$ and $L(D,K)$ as $M$ and $L(D)$ in \eqref{math:M} and \eqref{math:L}, with $A$ and $B$ replaced by $A_K$ and $B_d$, respectively.

Next, we reformulate Lemma~\ref{lem:KYP} for the closed-loop system.
\begin{prop}
   \label{prop:feas} 
   The system \eqref{math:lin_sys_} with the  state feedback controller $u=-Kx$, resulting in \eqref{math:closed_lin}, is (strictly) passive if and only if there exist matrices $Z\succ 0$ and $W\in\mathbb{R}^{m\times n}$ such that for $D\neq 0$
   \begin{align}
    \underbrace{\begin{bmatrix}Z A^\top + AZ - W^\top B_u^\top - B_u W & B_d - ZC^\top \\ B_d^\top - C Z & -D-D^\top \end{bmatrix}}_{\eqqcolon N_D(Z,W)} \preceq_s 0 \label{math:KYP_K} 
   \end{align} 
   and for $D=0$
   \begin{align}
    \underbrace{Z A^\top + AZ - W^\top B_u^\top - B_u W}_{\eqqcolon N_0(Z,W)} \preceq_s 0 \; \; \text{and}\;  \;  B_d^\top = C Z  \label{math:KYP_K_} 
   \end{align} 
   hold. Then, the feedback gain is $K = W Z^{-1}$ and the corresponding storage function is $S(x) = \frac{1}{2}x^\top Z^{-1} x$.
\end{prop}
\begin{proof}
   Let $Z = P^{-1}$ and $W = KZ$. For $D=0$, applying the congruent transformation $Z$ to $A_K^\top P + PA_K \preceq_s 0$ and right-multiplying $B_d^\top P = C$ by $Z$ yields \eqref{math:KYP_K_}. For $D\neq 0$, applying the congruent transformation $\operatorname{diag}(Z,I_p)$ to $M_K\preceq_s 0$ yields \eqref{math:KYP_K}. Since $Z\succ 0$ is invertible, all steps are reversible with $P = Z^{-1}$ and $K = WZ^{-1}$.
\end{proof}

Proposition~\ref{prop:feas} extends Lemma~\ref{lem:KYP} to characterize feedback gains $K$ that render the closed-loop system (strictly) passive. The change of variables $(P,K)\to (Z,W)=(P^{-1}, KP^{-1})$ convexifies the conditions \eqref{math:L} to \eqref{math:KYP_K} and  \eqref{math:KYP_K_}. 
The sets of stabilizing, passivating, and strictly passivating feedback gains are defined as
\begin{align}
   \mathcal{K} & \coloneq \{ K \in \mathbb{R}^{m \times n} \mid \operatorname{Re}(\lambda_i(A-B_uK)) < 0 ~ \forall i \} \label{math:K_} \\
   & ~= \{K \in \mathbb{R}^{m \times n} \mid \exists X\succ 0: A_K^\top X + X A_K \prec 0\}, \label{math:K__} \\
   \mathcal{P} & \coloneq \{ K \in \mathbb{R}^{m \times n} \mid
   \exists P \succ 0: L_{\text{ns}}(D,K)
   \}, \label{math:P_} \\
   \mathcal{P}^+ \! & \coloneq\{ K \in \mathbb{R}^{m \times n} \mid \exists P\succ 0 : L_{\text{s}}(D,K)
   \}, \label{math:P_+} 
\end{align}
and are referred to as the stability region and (strict) passivity region, respectively. The set $\mathcal{P}^{(+)}\in\{\mathcal{P}, \mathcal{P}^+\}$ is used to represent $\mathcal{P}$ for passivity or $\mathcal{P}^+$ for strict passivity, depending on the context.
Note that $\mathcal{K}\neq \emptyset$ due to the stabilizability of $(A,B_u)$ as stated in Assumption~\ref{ass:c+o}.

The following Lemma establishes the inclusion relations among the sets introduced above.
\begin{lem}
   \label{lemma:relation} 
   It holds that $\mathcal{P}\subseteq  \operatorname{cl}(\mathcal{K})=\mathcal{K} \cup \partial \mathcal{K}$ and $\mathcal{P}^+\subseteq \mathcal{K}$.
\end{lem}

\begin{proof}
   First, consider the case $D=0$.
   By definition, for all $K\in  \mathcal{P}^{(+)}$, there exists a $P\succ 0$ such that $A_K^\top P + P  A_K \preceq_s 0$. Thus, $A_K$ is at least marginally stable for passivity and asymptotically stable for strict passivity, which implies $\mathcal{P}\subseteq \operatorname{cl }(\mathcal{K})$ and $\mathcal{P}^{+ }\subseteq \mathcal{K}$. Now, consider the case $D\neq 0$. According to \cite[Observation~7.1.2]{horn2012matrix}, all principal submatrices of a negative (semi-)definite matrix are negative (semi-)definite. Thus, for any $K\in\mathcal{P}^{(+)}$, $M_K \preceq_s 0$ implies $A_K^\top P + P  A_K \preceq_s 0$, 
   and the same conclusions as in the case $D=0$ follow.  
\end{proof}

It is well established that (strict) passivity implies (asymptotic) stability for $d=0$, as stated in \cite[Lemmas~6.6--6.7]{khalil_nonlinear_2002}.
Lemma~\ref{lemma:relation} restates this property for a family of closed-loop systems parametrized by the feedback~$K$.

With the introduced notation, the controller sought in Objective~\ref{obj} coincides with the solution of
\begin{subequations}
   \label{math:opt_1} 
   \begin{align}
   \min_{K,X_K} ~& \operatorname{tr}\left( X_K \right) \\
   \text{s.t.}~& A_K^\top X_K + X_K A_K + Q + K^\top R K = 0, \label{math:c1} \\
   & K\in \mathcal{K}\cap \mathcal{P}^{(+)}, 
   \end{align} 
\end{subequations}
where the LQR cost \eqref{math:cost_K} of the system $(A,B_u)$ is minimized subject to the passivity constraint $K \!\in \!\mathcal{P}^{(+)}$.\footnote{Note that $K\in\mathcal{K}$ implies that the solution $X_K$ of \eqref{math:c1} is unique and positive semidefinite, leading to bounded objective value.}
Since the sets $\mathcal{K}$ and $\mathcal{P}^{(+)}$ are challenging to characterize explicitly, the problem \eqref{math:opt_1} is not tractable in the current form.


We assume the following for well-posedness and feasibility of \eqref{math:opt_1}, which holds throughout this work.

\begin{assumption}
   \label{ass:3} 
    The intersection $\mathcal{K}\cap \mathcal{P}^{(+)}\neq \emptyset$ is nonempty.
\end{assumption}

\begin{remark}
   \label{remark:int} 
   Assumption~\ref{ass:3} can be verified with Proposition~\ref{prop:feas} for strict passivity since $\mathcal{P}^+ \neq \emptyset$ implies $\mathcal{K}\cap \mathcal{P}^{+}\neq \emptyset$ due to Lemma~\ref{lemma:relation}.    
   For nonstrict passivity, Proposition~\ref{prop:feas} only ensures $\operatorname{cl}(\mathcal{K}) \cap \mathcal{P} \neq \emptyset$.   
   If $K \in \partial \mathcal{K}$, $A_K$ is marginally stable, causing the LQR cost \eqref{math:cost_K} to be unbounded, as the integral \eqref{math:int_P} diverges. In this case, feasibility holds, but the problem is ill-defined due to unboundedness.
   
\end{remark}

Next, we present two explicit formulations of problem \eqref{math:obj_prob} or \eqref{math:opt_1}. Combining the LQR problems \eqref{math:LQR_feron}, \eqref{math:LQR_boyd}, and the passivity constraint $L(D,K)$ yields
\begin{subequations}
   \label{math:nonconv_1} 
\begin{align}
    \min_{K,Y,P}~ &   \operatorname{tr }(QY) + \operatorname{tr }(K^\top R K Y)\\
    \text{s.t.}~ & (A-B_u K) Y + Y (A-B_u K)^\top + I_n \preceq 0, \label{math:nonconv_1_c1}  \\
    &  L(D,K) ,  \label{math:nonconv_1_c2}\\
    & Y\succeq 0, ~ P\succeq 0,
\end{align}
\end{subequations}
\vspace{-0.02cm}and 
\begin{subequations}
     \label{math:nonconv_2} 
\begin{align}
    \max_{X,P,K} ~& \operatorname{tr }(X) \\
    \text{s.t.}~& \begin{bmatrix} A^\top X + X A + Q & XB_u \\ B_u^\top X & R \end{bmatrix} \succeq 0,\label{math:nonconv_2_c1} \\
    &  L(D,K), \label{math:nonconv_2_c2}  \\
   & K = R^{-1} B_u^\top X,  \label{math:nonconv_2_c3} \\
    & X\succeq 0,  ~P\succeq 0,
\end{align}
\end{subequations}
where the constraint $L(D,K)$ depends on the matrix $P$. Problem \eqref{math:nonconv_1} has the same optimal value as \eqref{math:opt_1}, whereas \eqref{math:nonconv_2} only maximizes the lower bound $\operatorname{tr}(X)$ on the LQR cost, so that the optimal values of \eqref{math:nonconv_2} and \eqref{math:opt_1} coincide if $K^*\in\mathcal{P}^{(+)}$.
The next two propositions establish that neither problem \eqref{math:nonconv_1} nor \eqref{math:nonconv_2} can be reformulated as a convex problem.

\begin{prop}
   \label{prop:nonconv1} 
   The optimization problem \eqref{math:nonconv_1} admits no convex reformulation.
\end{prop}

\begin{proof}
   Constraints \eqref{math:nonconv_1_c1} and \eqref{math:nonconv_1_c2} are BMIs and coupled through the same controller gain $K$. Introducing new variables $F=KY, Z=P^{-1}$, and $W = K Z$ yields LMIs in $(Y,F)$ and $(Z,W)$, respectively. However, consistency requires $F Y^{-1} = W Z^{-1}$, equivalent to the bilinear matrix equality (BME) $FZ = W Y$. This equality is inherently nonconvex and cannot be convexified while maintaining problem equivalence. Although $ \operatorname{tr}\left(K^\top R K Y \right)$ admits a convex reformulation via Schur complement (\cite{Feron}), the essential nonconvexity persists due to the inseparable bilinear coupling between the constraints.   
\end{proof}

\begin{prop}
   The optimization problem \eqref{math:nonconv_2} admits no convex reformulation.
\end{prop}
\begin{proof}
   Introducing the variables $Z=P^{-1}, W = KZ$, transforms \eqref{math:nonconv_2_c2} into an LMI. Then, the equality \eqref{math:nonconv_2_c3} results in a BME $RW =  B_u^\top X Z$ that cannot be convexified. On the other hand, substituting $K = R^{-1} B_u^\top X$ into \eqref{math:nonconv_2_c2} still results in a BMI that contains bilinear terms $ZX$ and $XZ$. 
   These terms cannot be substituted by a new variable $S=ZX$ without losing the coupling of  \eqref{math:nonconv_2_c1} and \eqref{math:nonconv_2_c2} or introducing a new BME $S=ZX$.
\end{proof}

In summary, although the LQR problem and passivity constraints each admit a convex reparameterization individually, their coupling through the common feedback gain 
$K$ introduces bilinearities, rendering \eqref{math:nonconv_1} and \eqref{math:nonconv_2} inherently nonconvex.   
Due to these bilinearities, multiple optimal feedback gains may exist.

\subsection{Topological Properties of the Passivity Region}
Next, we explore the topological properties of the set $\mathcal{P}^{(+)}$. Denote the convex reparameterizations of the solution set of $L(D,K)$ for  the cases $D\neq 0$ and $D =0$ as 
\begin{align}
 \mathcal{L}^{(+)}_D & \!\coloneq \{Z\in\mathbb{R}^{n\times n}, W\in\mathbb{R}^{m\times n}\!\mid \!Z\succ0, \eqref{math:KYP_K} \text{ holds}\}, \\
  \mathcal{L}^{(+)}_0 & \!\coloneq \{Z\in\mathbb{R}^{n\times n}, W\in\mathbb{R}^{m\times n}\!\mid\! Z\succ0,\eqref{math:KYP_K_}\text{ holds}\}. 
\end{align}

\begin{lem}
   The set $\mathcal{P}^{(+)}$ is unbounded.
\end{lem}
\begin{proof}
   First, consider the case $D\neq 0$. Since $\mathcal{P}^{(+)}$ is nonempty by Assumption~\ref{ass:3}, fix $K_0 \in \mathcal{P}^{(+)}$ with associated $P_0$ satisfying $M_{K_0}\preceq_s 0$. Let $Z_0 = P_0^{-1}$ and  $W_0 = K_0 Z_0$. For $t \geq 0$, define $W(t) = W_0 + t B_u^\top$. Then $(Z_0, W(t)) \in \mathcal{L}_D^{(+)}$ for all $t\geq 0$ because $Z_0 \succ 0$ and 
  \begin{align*}
    &\begin{bmatrix}Z_0 A^\top + AZ_0 - W(t)^\top B_u^\top - B_u W(t) & B_d - Z_0C^\top \\ B_d^\top - C Z_0 & -D-D^\top \end{bmatrix} \\
    &= \begin{bmatrix}Z_0 A^\top + AZ_0 - W_0^\top B_u^\top - B_u W_0 & B_d - Z_0C^\top \\ B_d^\top - C Z_0 & -D-D^\top \end{bmatrix} \\&\quad  - 2 t \begin{bmatrix}B_u B_u^\top & 0 \\ 0 & 0\end{bmatrix} \preceq_s 0.
   \end{align*} 
   The corresponding gain $K(t) = W(t) Z_0^{-1} = K_0 + t B_u^\top P_0\in \mathcal{P}^{(+)}$ satisfies $\Vert K(t) \Vert_F \to \infty$ as $t \to \infty$ because $B_u^\top P_0 \neq 0$ since $B_u\neq 0$ and $P_0\succ 0$. Thus, $\mathcal{P}^{(+)}$ is unbounded. The proof for the case $D=0$ follows analogously. 
\end{proof}

\begin{lem}
The set $\mathcal{P}^{(+)}$ is path-connected.
\end{lem}
\begin{proof}
   Due to Assumption~\ref{ass:3}, $\mathcal{P}^{(+)}\neq \emptyset$ and consequently $\mathcal{L}_{D}^{(+)}\neq  \emptyset$ and $ \mathcal{L}_{0}^{(+)} \neq  \emptyset$. Both sets are convex since each is defined by an affine LMI, and $\mathcal{L}_0^{(+)}$ additionally includes an affine matrix equality. For  $D\neq 0$, define the map  $\phi_D: \mathcal{L}_D^{(+)} \to \mathcal{P}^{(+)}, \phi_D(Z,W) = W Z^{-1}=K$ which is continuous in $(Z,W)$ since $Z\succ 0$.
   Convexity of $\mathcal{L}_D^{(+)}$ implies it is path-connected, and continuity of $\phi_D$ preserves path-connectedness, so $\mathcal{P}^{(+)}$ is path-connected.
   The proof for the case $D=0$ follows analogously.
\end{proof}

\begin{lem}
   The set $\mathcal{P}^{(+)}$ is contractible. 
\end{lem}
\begin{proof}
Consider the case $D\neq 0$ and the map $\phi_D: \mathcal{L}_D^{(+)} \to \mathcal{P}^{(+)}$ given by $\phi_D(Z,W) = W Z^{-1}=K$. By definition, the map $\phi_D$ is surjective. We construct a continuous section $s: \mathcal{P}^{(+)} \to \mathcal{L}_D^{(+)}$ such that $\phi_D(s(K))=K$ for all $K\in\mathcal{P}^{(+)}$.  For each $K \in \mathcal{P}^{(+)}$, the fiber $\phi_D^{-1}(K)$ is convex and nonempty. 
Let $F(K) = \{Z\succ 0 \mid N_D(Z,KZ)\preceq_s 0\}$, $J(Z) = \Vert Z\Vert_F^2 - \log\det (Z)$, with the additional term $-\log\det(-N_D(Z,KZ))$ for strict passivity, and $Z(K) = \arg\min_{Z\in F(K)}  J(Z)$.
The function $J(Z)$ is strictly convex and coercive. For $K\in\mathcal{P}^{(+)}$, coercivity guarantees the existence of the minimizer  $Z(K)$ despite $F(K)$ being unbounded, while strict convexity ensures the uniqueness of $Z(K)$.
 The continuity of $K\mapsto Z(K)$ follows from parametric convex programming since $F(K)$ varies continuously with $K$ and $J$ is continuous. Thus, $s(K)=(Z(K),KZ(K))$ is a continuous section. Define the homotopy $H:\mathcal{P}^{(+)}\times [0,1]\to\mathcal{P}^{(+)},~H(K,t) = \phi_D((1-t)s(K) + ts(K_0))$ for a fixed $K_0 \in \mathcal{P}^{(+)}$. The convexity of $\mathcal{L}_D^{(+)}$ ensures $H(K,t) \in \mathcal{P}^{(+)}$ for all $t \in [0,1]$, with $H(K,0) = K$ and $H(K,1) = K_0$. Thus $H$ continuously contracts $\mathcal{P}^{(+)}$ to $K_0$, proving contractibility. The case $D=0$ can be proved in the same manner by incorporating the additional linear constraint $B_d^\top = C Z$ into $F(K)$.
\end{proof}

\begin{lem}
There exist matrices $(A,B_u,B_d,C,D)$ defining the system \eqref{math:lin_sys_} for which the set $\mathcal{P}^{(+)}$ is nonconvex.
\end{lem}
\begin{proof}
    The set $\mathcal{P}^{(+)}$ is the image of the convex sets $\mathcal{L}_D^{(+)}$ or $\mathcal{L}_0^{(+)}$ under the nonlinear map $(Z,W) \mapsto WZ^{-1}$ for $D\neq 0$ and $D=0$, respectively. Since this map is not affine, its image is not guaranteed to be convex. For a concrete example, consider the set $\mathcal{P}$ \eqref{math:P_} of the system 
\begin{align}
A = I_2,~
B_u = I_2,~
B_d =\begin{bmatrix}1 \\ 0\end{bmatrix},~
C= \begin{bmatrix}1 & 0\end{bmatrix},~
D =0.
\end{align}
The constraints $B_d^\top P = C$ and $P\succ 0$ force $P = \begin{bmatrix}1 & 0 \\ 0 & p\end{bmatrix}$ with $p \!> \!0$, and  $A_K^\top P\! + \!P A_K\!\preceq \!0$ with $K \!=\! \begin{bmatrix} K_1 & K_2 \\ K_3 & K_4\end{bmatrix}$ yields 
\begin{align}
   \label{math:nonconvex} 
\begin{bmatrix}
2(1- K_1) & -K_2-pK_3 \\
-K_2-pK_3 & -2p(K_4-1)
\end{bmatrix} \preceq 0.
\end{align}
By Sylvester's criterion, \eqref{math:nonconvex} is equal to  $K_1 \geq 1$, $K_4 \geq 1$, and $4p(1-K_1)(1-K_4)-(K_2+p K_3)^2 \geq 0$ for some $p > 0$. The gains $\tilde K_1 = \begin{bmatrix}2 & 2 \\ 0.5 & 2\end{bmatrix}$ and $\tilde K_2 = \begin{bmatrix}2 & 0.5 \\ 2 & 2\end{bmatrix}$ are passivating with $p = 4$ and $p = \tfrac{1}{4}$, respectively. However, their convex combination $K_c = \frac{1}{2}\tilde K_1 + \frac{1}{2} \tilde K_2 = \begin{bmatrix}2 & 1.25 \\ 1.25 & 2\end{bmatrix}$ requires $\frac{7p}{8}-\frac{25p^2}{16}-\frac{25}{16} = - \frac{25}{16} \bigl( p- \frac{7}{25} \bigr)^2 - \frac{36}{25}\geq 0$ which fails to hold for any $p>0$. Thus, $K_c$ is not passivating, proving the nonconvexity of the set $\mathcal{P}$.  
\end{proof}
An example of a convex set $\mathcal{P}$ is given in Section~\ref{sec:simu}.

\section{Indirect Approach}
To address the intractability of problems \eqref{math:nonconv_1} and \eqref{math:nonconv_2}, we propose an indirect method for computing feasible solutions.
Specifically, we present a three-step procedure that computes a feasible feedback gain for Objective~\ref{obj} and yields an upper bound on the optimal value of~\eqref{math:opt_1}. The approach is illustrated in Fig.~\ref{fig:4step}.  
In the first step, the passivity region is approximated as a union of convex and compact polytopes.  In the second step, this union is approximated from the inside by a single convex polytope. 
Finally, in the third step, the state-feedback gain within this convex polytope that minimizes the LQR cost is computed.

\begin{figure}[t]
\hspace{-0.3cm}
\begin{tikzpicture}[
  rectbox/.style={
    draw,
    rounded corners,
    minimum width=2.4cm,
    minimum height=1.8cm, 
     fill=lightsteelblue2,  
    fill opacity=0.3,   
  }, font=\scriptsize,arrow style/.style={-{Latex}}
]

\node[rectbox,anchor=west] (r1) at (0,0) {};
\node[rectbox,right=6mm of r1] (r2)  {};
\node[rectbox,right=6mm of r2] (r3)  {};

\node at (r1.north) [below=1pt] {Step 1: Sec.~\ref{sec:4.1} };
\node at (r2.north) [below=1pt] {Step 2: Sec.~\ref{subsec:4_2}};
\node at (r3.north) [below=1pt] {Step 3: Sec.~\ref{subsec:con}};

\foreach \r in {r1,r2,r3} {
  \draw ($( \r.north west)!0.25!(\r.south west)$) --
        ($( \r.north east)!0.25!(\r.south east)$);
}

\node[align=center, text width=3cm, yshift=-6pt]  at (r1.center) {Characterizing the  passivity region: \par  Lemma~\ref{lem:vertices} \par 
$\rightarrow \mathcal{P}_U\subset \mathcal{P}^{(+)}$};
\node[align=center, text width=3cm, yshift=-6pt]   at (r2.center) {Convex inner approximation: Lemma~\ref{lem:under} \par $\rightarrow \mathcal{P}_C\subseteq \mathcal{P}_U$};
\node[align=center, text width=3cm, yshift=-6pt] at (r3.center) {Constraint Optimization: Theorem~\ref{theo:1} \par $\rightarrow \hat{K}\in\mathcal{P}_C$};
\draw[arrow style, thick] (r1.east) -- (r2.west);
\draw[arrow style, thick] (r2.east) -- (r3.west);
\end{tikzpicture}
\vspace{-12pt}
\caption{Overview of the  three steps of the indirect approach}
\label{fig:4step}
\end{figure}
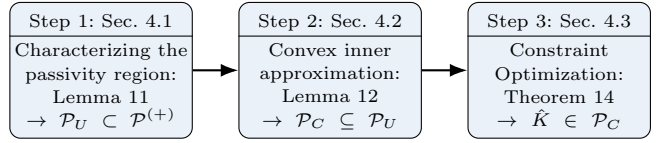

\begin{remark}
   \label{remark:2} 
   Before proceeding, it is worth verifying whether the optimal LQR feedback $K^*$ already renders the system \eqref{math:lin_sys_} (strictly) passive. This can be checked using Lemma~\ref{lem:KYP}.  If this condition holds, then Objective~\ref{obj} is already satisfied. Otherwise, the following method is proposed
\end{remark}

\subsection{Characterizing the Passivity Region}
\label{sec:4.1} 
In this subsection, we  state a sufficient condition such that all gains within a convex polytope render the closed-loop system (strictly) passive.

\begin{lem}
   \label{lem:vertices} 
   Let $\mathcal{C}\subset \mathbb{R}^{mn}$ be a convex compact polytope. 
   If there exists a matrix $P\succ 0$ such that  $    L(D,K)$ holds for all $K \in \operatorname{vert}(\mathcal{C})$, then every $K\in\mathcal{C}$ (strictly) passivates the system \eqref{math:lin_sys_} with the common storage function $S(x) = \frac{1}{2}x^\top P x$. Consequently $\mathcal{C}\subseteq \mathcal{P}^{(+)}$.
\end{lem}

\begin{proof} 
   Let $\operatorname{vert }(\mathcal{C}) = \{K^{0}, \dots, K^{r} \}$. 
   Every $K\in \mathcal{C}$ is a convex combination $K=\sum_{i=0}^{r}\lambda_i K^{i}$ with $\lambda_i\ge 0$ and $\sum_{i=0}^{r}\lambda_i=1$. Consider first the case $D\neq 0$. Since $M_K$ is affine in $K$, we have $M_K = \sum_{i=0}^{r} \lambda_i M_{K^i} \preceq_s 0$. Hence, $M_K\preceq_s 0$ for all $K\in\mathcal{C}$, showing that $S(x)=\tfrac{1}{2}x^\top P x$ is a common storage function, and $\mathcal{C}\subseteq \mathcal{P}^{(+)}$. 
   The same reasoning applies for $D = 0$.
\end{proof}

Lemma~\ref{lem:vertices} establishes that the existence of a common storage function for all gains $K \in \mathcal{C}$
can be verified through a finite number of LMIs, evaluated only at the vertices of $\mathcal{C}$.
This condition can be efficiently checked using standard robust optimization techniques for polytopic uncertainty, e.g., see \cite{lofberg2012}. In practice, the considered polytopes must be sufficiently small to guarantee that, if $\mathcal{C}\subset\mathcal{P}^{(+)} $, a common $P\succ 0$ satisfying 
Lemma~\ref{lem:vertices} exists. 

The key idea of the indirect approach is to construct an inner approximation of the passivity region $\mathcal{P}^{(+)}$ using a union of convex, compact, and connected polytopes $\mathcal{C}_i$ for which Lemma~\ref{lem:vertices} holds, i.e.,  $\mathcal{P}_{U} = \bigcup_i \mathcal{C}_i\subseteq \mathcal{P}^{(+)}$.
For simplicity and ease of alignment, hypercubes are employed in the following,
although the method naturally extends to other convex, compact polytopes such as simplices,
cross-polytopes, or parallelepipeds, which can be translated and  rotated to tile $\mathcal{P}^{(+)}$ efficiently.

\begin{figure}[tb]
    \centering
\includegraphics[width=0.97\columnwidth]{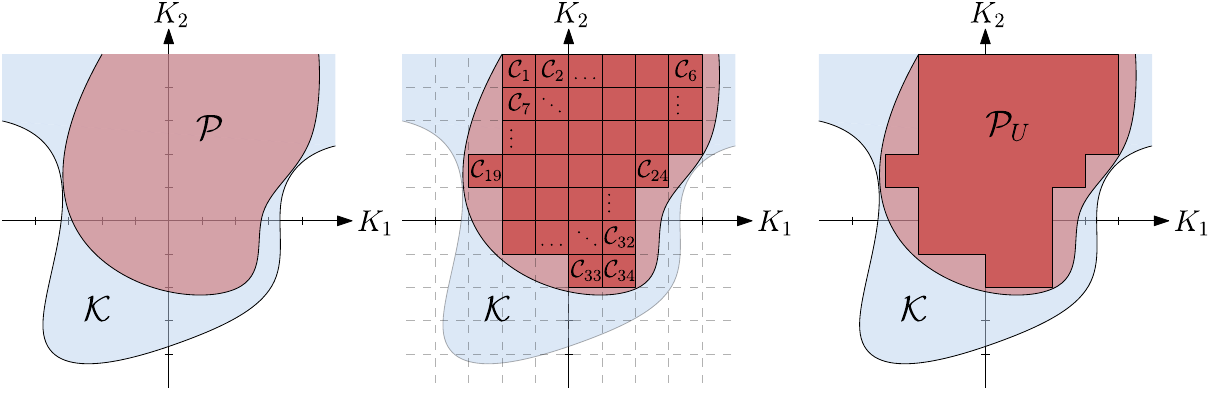}
 \caption{Illustration of the stability region $\mathcal{K}$,
the passivity region $\mathcal{P}$, the verified cubes $\mathcal{C}_i$,
and their union $\mathcal{P}_U$.}
 \label{fig:sketch}
\end{figure}

Fig.~\ref{fig:sketch} depicts the idea  for the special case of a 2-dimensional feedback matrix.
In this example,  Lemma~\ref{lem:vertices} is evaluated for 100 cubes, of which passivity is verified for 34.  
The union $\mathcal{P}_U$ of these cubes forms a compact inner approximation of the unbounded passivity region.
The use of multiple convex polytopes is motivated by the possible nonconvexity of $\mathcal{P}^{(+)}$, which prevents any single convex set from capturing it. Moreover, path-connectedness ensures that expanding the search via hypercubes adjacent to previously verified ones preserves the connectivity of $\mathcal{P}_U$.


\begin{remark}
Several practical considerations arise when implementing the cube-based search.
First, the search limits should ideally be near the optimal LQR gain.
To accelerate the process, one can explore hypercubes adjacent to those already verified, start with larger hypercubes, and refine the search with smaller ones in regions where feasibility fails.
A more accurate inner approximation can be obtained by using smaller cubes
 near the boundary of $\mathcal{P}^{(+)}$. 
\end{remark}

\subsection{Convex Inner Approximation of the Passivity Region}
\label{subsec:4_2} 
Although the union $\mathcal{P}_U$ provides a valid inner approximation of the passivity region $\mathcal{P}^{(+)}$, it may be nonconvex. 
This nonconvexity prevents the direct use of $\mathcal{P}_U$ in optimization when solving for the optimal LQR gain within the verified region.
Therefore, we seek a convex inner approximation of $\mathcal{P}_U$ by a single polytope.

\begin{lem}
   \label{lem:under} 
   Let  $\mathcal{P}_{U} = \bigcup_i \mathcal{C}_i\neq\emptyset$ be a finite union of full-dimensional, convex, compact and connected polytopes $\mathcal{C}_i\subseteq \mathcal{P}^{(+)}$.
   Then there exists a convex, compact polytope
$\mathcal{P}_C$ such that $\mathcal{P}_C\subseteq \mathcal{P}_U\subset \mathcal{P}^{(+)}$. 
\end{lem}

\begin{proof}
   Since $\mathcal{P}_U$ has nonempty interior, let $K_0\in \operatorname{int}(\mathcal{P}_U)$
   and $\varepsilon >0$ such that the open ball $B(K_0,\varepsilon)\subseteq \mathcal{P}_U$.
   Inside this ball, inscribe any convex compact polytope $\mathcal{P}_C$, e.g., a hypercube centered at $K_0$ with edge length $\varepsilon/(2\sqrt{mn})$ ensures $\mathcal{P}_C\subseteq B(K_0,\varepsilon)$ because the half-diagonal of the hypercube is $\varepsilon/2$. 
   By construction,  $\mathcal{P}_C\subseteq \mathcal{P}_U \subset  \mathcal{P}^{(+)}$.
\end{proof}

\begin{figure}[t]
    \centering
\includegraphics[width=1\columnwidth]{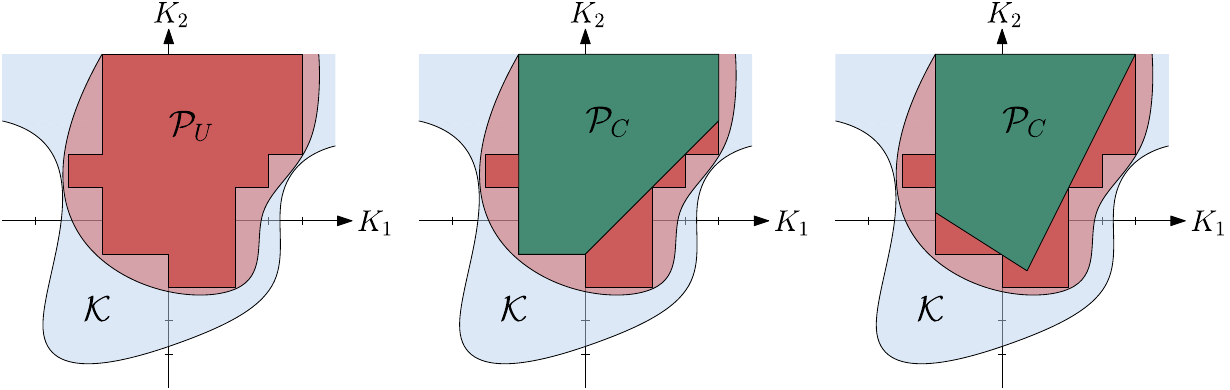}
 \caption{Illustration of the stability region $\mathcal{K}$,
the passivity region $\mathcal{P}$, the union $\mathcal{P}_U$ of the verified cubes $\mathcal{C}_i$ and their convex inner approximation by $\mathcal{P}_C$.}
 \label{fig:sketch2}
\end{figure}

Fig.~\ref{fig:sketch2} illustrates two convex inner approximations of $\mathcal{P}_U$ using a polygon $\mathcal{P}_C$. 
Ideally, the polygon should be selected to maximize its area, thereby minimizing the approximation error.
In Fig.~\ref{fig:sketch2}, the middle polygon is about $8\%$ larger than the right one.

\begin{remark}
   \label{remark:inner} 
Obtaining the largest $\mathcal{P}_C $ such that $\mathcal{P}_C\subseteq \mathcal{P}_U$ is a nontrivial problem.
When $mn=2$, the problem of finding the convex polygon of the largest possible area that lies within a given nonconvex simple polygon is known as the \textit{potato peeling} problem.
In two dimensions, the problem can be efficiently solved by the algorithm from \cite{chang1986polynomial}.
In higher dimensions, to the best of the authors' knowledge,  no efficient algorithms exist. 
\end{remark}

For the subsequent subsection, we introduce the notation 
\begin{align}
 \mathcal{P}_C=\{ K\in\mathbb{R}^{m\times n}\mid g(\operatorname{vec}\left(K \right))\geq 0\}\subset \mathbb{R}^{m\times n},
\end{align} 
where $g:\mathbb{R}^{m n}\to\mathbb{R}^l$ is affine,  $J_g\in\mathbb{R}^{l\times mn}$ denotes its constant Jacobian, and $\mathcal{P}_C$ is a compact, convex polytope. We assume that LICQ holds at every point of $\mathcal{P}_C$, which is the case if $\mathcal{P}_C$ is simple, i.e., every vertex lies on exactly $mn$ facets.

\subsection{Constraint Optimization}
\label{subsec:con} 

In this subsection, we introduce a projected gradient flow to compute the optimal feedback gain $\hat{K}$ with respect to the LQR objective within the set $\mathcal{P}_C$. 
To this end, we incorporate the projection mechanism from Subsection~\ref{subsec:proj} into the policy gradient flow \eqref{math:grad}. The resulting projected gradient flow is defined as\footnote{We slightly abuse notation in the projection by identifying $\mathcal{P}_C\subset \mathbb{R}^{m\times n}$ with its vectorized counterpart in $\mathbb{R}^{mn}$.}
\begin{align}
 \operatorname{vec}(\dot K )& \!=\! - \alpha \Pi_{\mathcal{P}_C}  \! \operatorname{vec}( \nabla f_K ), ~ K (0) \in \operatorname{int}(\mathcal{P}_C), \label{math:proj_1} 
\end{align} 
where $\alpha>0$ is the learning rate, 
\begin{align}
 \Pi_{\mathcal{P}_C} & = I_{mn} - J_g^\top \bigl(2 \operatorname{diag}\bigl(g(\operatorname{vec}\left(K  \right))\bigr) + J_g J_g^\top \bigr)^{-1} \!J_g, \label{math:projection}  \\
 \nabla f_K & = 2 ( R K - B_u^\top X_K ) Y_K
\end{align} 
and $X_K$ and $Y_K$ are the solutions of the Lyapunov equations \eqref{math:Lyap1} and \eqref{math:Lyap2}, respectively.

\begin{lem} Let $\mathcal{P}_C\subset \mathcal{P}^{(+)}$. Then, the projected gradient flow  \eqref{math:proj_1} is well-defined for all $K\in\mathcal{P}_C$.
\end{lem}
\begin{proof}
   Since LICQ holds on $\mathcal{P}_C$, the matrix inverted in \eqref{math:projection} is positive definite for all $K\in\mathcal{P}_C$, so that $\Pi_{\mathcal{P}_C}$ is well-defined.
   The gradient $\nabla f_K$ is well-defined for all $K\in\mathcal{K}$. By Lemma~\ref{lemma:relation}, $\mathcal{P}^{(+)  }\subseteq   \mathcal{K} \! \cup \! \partial \mathcal{K}$, and therefore $\mathcal{P}_C \subset \mathcal{K} \!\cup\! \partial \mathcal{K}$.   
   The boundary $\partial \mathcal{P}_C$ may contain marginally stable gains, for which $\nabla f_K$ is undefined due to divergence of \eqref{math:int_P} or \eqref{math:Y}.
       By Proposition~\ref{prop:invariant}, $\operatorname{int}(\mathcal{P}_C)$ is a positive invariant set. Thus, $K (0)\in \operatorname{int}(\mathcal{P}_C)$ implies   
   $K(t) \in \operatorname{int}(\mathcal{P}_C)$ for all $t\geq 0$.   
  Hence, $K(t)$ never reaches $\partial\mathcal{P}_C$  and~\eqref{math:proj_1} remains well-defined for all $t \ge 0$.
\end{proof}

\begin{thm}
   \label{theo:1} 
   Let $\mathcal{P}_C\subset \mathcal{P}^{(+)}$.
   Then, the projected gradient flow \eqref{math:proj_1} solves the optimization problem\footnote{Because $\mathcal{P}_C\subseteq \mathcal{P}^{(+)}$, every solution of \eqref{math:opt_main} is feasible for \eqref{math:opt_1}, and the optimal value of \eqref{math:opt_main} upper-bounds that of \eqref{math:opt_1}.}
       \begin{align}
         \label{math:opt_main} 
    \hat{K}=\arg\min_{K}~ & f_K \quad   \text{s.t.} \quad K\in  \mathcal{K}\cap \mathcal{P}_C,
   \end{align} 
   where $f_K$ is defined as in \eqref{math:cost_K}. Specifically, the trajectory $K(t) $  converges to $\hat{K}$ and remains within $\mathcal{P}_C$.
\end{thm}

\begin{proof}
   Over $\mathcal{K}$, $f_K$ is gradient dominant and has a unique global minimizer $K^*$. Hence, the constraint problem \eqref{math:opt_main} has a unique minimizer $\hat{K}$, because $\mathcal{P}_C$ is convex. Define $V(K) = f_K - f_{\hat{K}}$. Then $V(K)\geq 0$ and $V(K)=0$ if $K=\hat{K}$. Differentiating $V$ along \eqref{math:proj_1} gives $\dot{V}= \frac{\partial f_K}{\partial \operatorname{vec}\left(K \right)} \operatorname{vec}(\dot K )= - \alpha \operatorname{vec}(\nabla f_K)^\top \Pi_{\mathcal{P}_C} \operatorname{vec }(\nabla f_K) \leq 0 $, because $\Pi_{\mathcal{P}_C}$ in \eqref{math:projection} is positive definite on $\operatorname{int}(\mathcal{P}_C)$ by the Sherman-Morrison-Woodbury formula \cite[Sec.~0.7.4]{horn2012matrix}, and $K(t)\in\operatorname{int}(\mathcal{P}_C)$ by Proposition~\ref{prop:invariant}. Equality holds only at $K=\hat{K}$ and $\dot V <0 $ for all $K\neq \hat{K}$. Hence, $V$ is a proper Lyapunov function and every trajectory $K(t)$ of \eqref{math:proj_1} converges to $\hat{K}$ as $t\to\infty$, since $V$ is monotonically decreasing  and every trajectory remains  $\mathcal{P}_C$ by Proposition~\ref{prop:invariant}.
\end{proof}

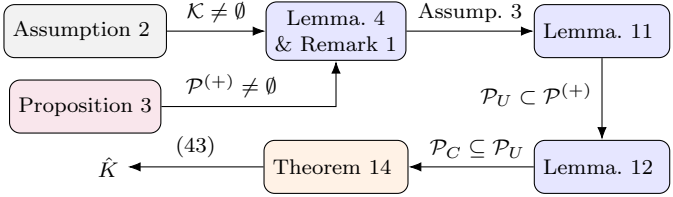
\begin{figure}[t]
\begin{tikzpicture}[
   node style/.style={rectangle, draw=black, align=center, minimum height=2em, minimum width=2em,rounded corners},
   lem style/.style={rectangle, draw=black, align=center, minimum height=2em, minimum width=3em,rounded corners,fill=blue!10},
    prop style/.style={rectangle, draw=black, align=center, minimum height=2em, minimum width=3em,rounded corners,fill=purple!10},
   theo style/.style={rectangle, draw=black, align=center, minimum height=2em, minimum width=2em,rounded corners,fill=orange!10},
   results style/.style={rectangle, draw=black, align=center, minimum height=2em, minimum width=2em,rounded corners,fill=gray!10},
   text style/.style={draw=none, fill=none,align=center},
   plus style/.style={circle, draw=black, inner sep=0.7pt, minimum size=0.6em, font=\scriptsize},
   arrow style/.style={-{Latex}},font=\footnotesize
 ]
  \node[results style] (A) at (0,0) {Assumption~\ref{ass:c+o} };
  \node[prop style] (B) at (0,-1) {Proposition~\ref{prop:feas}};
  \node[lem style] (C) at (3.3,0) {Lemma.~\ref{lemma:relation} \\ \&~Remark~\ref{remark:int}};
   \node[lem style] (D) at (6.8,0) {Lemma.~\ref{lem:vertices} };
   \node[lem style] (E) at (6.8,-1.8) {Lemma.~\ref{lem:under} };
   \node[theo style] (F) at (3.3,-1.8) {Theorem~\ref{theo:1}  };
   \node[text style] (G) at (0.3,-1.8) {$\hat{K}$};
\draw[arrow style] (A) --  node[midway, above] {$\mathcal{K}\neq \emptyset$} (C);
\draw[arrow style] (B) -|  node[pos=0.2, above] {$\mathcal{P}^{(+)}\neq \emptyset$} (C.south);
\draw[arrow style] (C) -- node[midway, above] {Assump.~\ref{ass:3}} (D);
\draw[arrow style] (D) -- node[pos=0.45, left] {$\mathcal{P}_U\subset\mathcal{P}^{(+)}$} (E);
\draw[arrow style] (E) -- node[pos=0.45, above] {$\mathcal{P}_C\subseteq \mathcal{P}_U$} (F);
\draw[arrow style] (F) -- node[midway, above] {\eqref{math:proj_1}} (G);
\end{tikzpicture}
\vspace{-0.5cm}
\caption{Structured overview of the logical dependencies of the indirect approach.}
\label{fig:5} 
\end{figure}

Several aspects are worth noting. First, the  trajectory of the projected gradient flow \eqref{math:proj_1} remains within $\mathcal{K}$, since $f_K$ is coercive over $\mathcal{K}$ and the initial condition satisfies $K(0) \in \mathcal{K}$.
Thus, \eqref{math:proj_1} implicitly satisfies the constraint $K\in\mathcal{K}$ in \eqref{math:opt_main}, even if $\partial \mathcal{K}\cap \mathcal{P}\neq \emptyset$.
Second, the constraint $K\in\mathcal{P}_C$ cannot be directly employed as a constraint in the optimization problems \eqref{math:LQR_feron} and \eqref{math:LQR_boyd}, as these problems do not admit a convex reparameterization in terms of $K$. Third, Objective~\ref{obj} is not satisfied exactly. The level of conservatism depends on the approximation error of $\mathcal{P}_C$ with respect to $\mathcal{P}^{(+)}$. Finally, if $\hat{K} \in \partial \mathcal{P}_C$, the trajectory $K(t)$ of \eqref{math:proj_1} does not reach $\hat{K}$ but can approach it arbitrarily close due to Proposition~\ref{prop:invariant} and $K(0)\in \operatorname{int}(\mathcal{P}_C)$.
Fig.~\ref{fig:5} summarizes the indirect approach by depicting the logical dependencies among assumptions, propositions, lemmas, and theorems.

\section{Simulation Results}
\label{sec:simu}
This section illustrates the indirect approach for an open-loop stable and an open-loop unstable system and evaluates their negative feedback interconnection.

\subsection{Open-Loop Stable System}
\label{subsec:ex1}
Consider the open-loop stable system
\begin{subequations}
   \label{math:sys_example} 
   \begin{align}
 A &= \begin{bmatrix} -2 & -1 \\ -1 & -3\end{bmatrix}, ~  B_u = \begin{bmatrix}1 \\ 2\end{bmatrix}, ~ B_d = \begin{bmatrix}2 \\ 1 \end{bmatrix}, ~ C = \begin{bmatrix}0 & 1\end{bmatrix},\\
 D & = 0,~ Q  = I_2 , ~ R =2, ~ K = \begin{bmatrix} K_1 & K_2 \end{bmatrix}.
\end{align} 
\end{subequations}
By the Routh-Hurwitz Criterion, the stability region~\eqref{math:K_}~is 
\begin{align}
 \mathcal{K}\!=\!\{ K \!\in\!\mathbb{R}^{1\times 2}\!\mid \!5 \!+ \!K_1\! + \!2 K_2 \!>\! 0,~ 5 \!+\! K_1\! +\! 3 K_2 \!>\!0 \}. \label{math:stab_example} 
\end{align} 
Next, the region $\mathcal{P}$ is calculated analytically for the sake of illustration. The constraints $B_d^\top P = C$ and $P\succ 0$ yield
\begin{align}
 P = \begin{bmatrix} p & -2 p \\ -2p & 1+4p\end{bmatrix},\quad p>0.
\end{align} 
Substituting this into $A_K^\top P + P A_K \preceq 0$ yields 
\begin{align}
 \tilde{M} = \begin{bmatrix} 6K_1 p~ & 5p - 2K_1 - 6 K_1 p + 3 K_2 p - 1 \\ * & - 4 K_2 - 20 p - 12 K_2 p - 6\end{bmatrix}\preceq 0. \label{math:M_an} 
\end{align} 
By Sylvester's criterion, \eqref{math:M_an} is equivalent to 
\begin{align}
   \label{math:set} 
 K_1 \leq 0, \quad 4K_2 + 20p+12K_2 p + 6 \geq 0, \quad \operatorname{det}(\tilde{M}) \geq 0.
\end{align} 
With $p>0$, the second inequality in \eqref{math:set} implies $K_2 > -\tfrac{5}{3}$.
The determinant has the form $\operatorname{det }(\tilde{M}) = c_1 p^2 + c_2 p + c_3 $, where $c_1 = -(6K_1+3K_2+5)^2 \leq 0$. On the boundary $\operatorname{det}(\tilde{M})=0$, the maximum of $\operatorname{det }(\tilde{M})$ over $p>0$ is zero. This occurs when $ \operatorname{det}(\tilde{M})=0$ has a double root in $p$, i.e., the discriminant $c_2^2 - 4 c_1 c_3 =0$ is zero, leading to the zeros $K_1=0$ and $ K_1 = - \tfrac{25 + 14 K_2 \pm 5\vert 3 + 2 K_2 \vert}{8}$. At the corner $K = \begin{bmatrix} -\tfrac{1}{2} & -\tfrac{3}{2} \end{bmatrix}$, however, $\operatorname{det}(\tilde{M}) = -\tfrac{25}{4}p^2 < 0$ for all $p>0$. Altogether
   \begin{align}
 \mathcal{P}&=\Bigl\{ K\in\mathbb{R}^{1\times 2}\mid K_1\leq 0, K_2 > - \tfrac{5}{3},~ \text{and} \notag \\
 &\quad \quad K_1 \geq -5 -3 K_2 ~\text{if}~ \tfrac{-5}{3}< K_2 < \tfrac{-3}{2}, \label{math:passivty_example} \\
 & \quad \quad K_1 \geq  -\tfrac{5}{4}-\tfrac{1}{2}K_2 ~\text{if}~ K_2 \geq \tfrac{-3}{2} \Bigr\} \notag \\
 & \quad \setminus \bigl\{ \begin{bmatrix} -\tfrac{1}{2} & -\tfrac{3}{2} \end{bmatrix} \bigr\} \notag
\end{align}
  \begin{figure}[t]
    \hspace{0.21cm}
\includegraphics[scale=0.85]{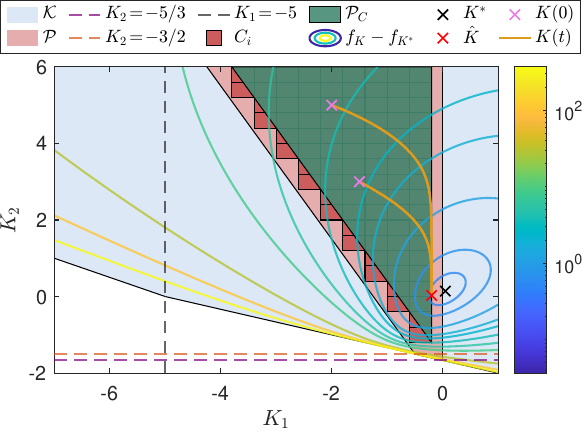}
 \caption{Visualization of stability region $\mathcal{K}$, the passivity region $\mathcal{P}$, the verified cubes $\mathcal{C}_i$, the convex inner approximation $\mathcal{P}_C$, the cost function $f_K - f_{K^*}$ and the trajectory of the projected gradient flow $K(t)$.}
 \label{fig:exam2}
\end{figure}
Fig.~\ref{fig:exam2} illustrates the stability region \eqref{math:stab_example} and the passivity region \eqref{math:passivty_example} of system \eqref{math:sys_example}. Consistent with Lemma~\ref{lemma:relation}, marginally stable feedback gains are part of the passivity region, i.e., gains that satisfy  $K_1 = -5-3 K_2$.  
The optimal LQR feedback of $(A,B_u)$ is $K^* \approx \begin{bmatrix} 0.048 & 0.143\end{bmatrix}$ and lies outside the passivity region since $K_1^*>0$.
Fig.~\ref{fig:exam2} further depicts the verified cubes $\mathcal{C}_i\subset \mathcal{P}$ with edge length $0.4$, the resulting convex inner approximation $\mathcal{P}_C$, the normalized cost function $f_K - f_{K^*}$, and the trajectories of the projected gradient flow \eqref{math:proj_1} for two distinct initial values satisfying $K(0)\in \operatorname{int}(\mathcal{P}_C)$. The trajectories $K(t)$ converge to $\hat{K}$ within the set $\mathcal{P}_C$ that minimizes the LQR cost. 
  \begin{figure}[t]
    \hspace{0.24cm}
\includegraphics[scale=0.87]{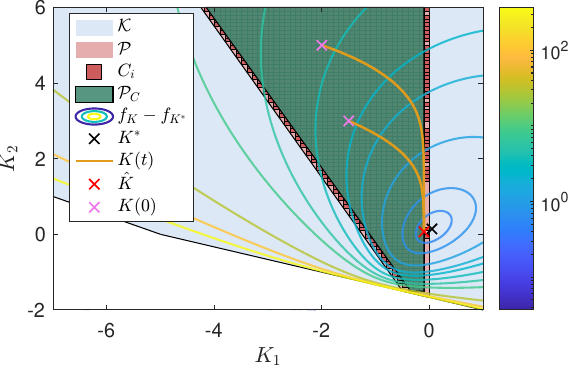}
 \caption{Visualization analogous to Fig.~\ref{fig:exam2}, with cubes $\mathcal{C}_i$ having edge length $0.1$.  }
 \label{fig:exam3}
\end{figure}

The effect of refining the grid is illustrated in Fig.~\ref{fig:exam3}, where the cube edge length is reduced to $0.1$. The inner approximation $\mathcal{P}_C$ is now characterized by four inequalities, and $\hat K \approx \begin{bmatrix} -0.100 & 0.069 \end{bmatrix}$.
For $ -0.1 \leq K_1 \leq 0$ and $K\in\mathcal{P}$, some cubes $\mathcal{C}_i$ within the passivity region could not be verified with Lemma~\ref{lem:vertices}.
It remains unclear whether this issue is due to numerical inaccuracies or whether the cube is too large for the existence of a common storage function. Nonetheless, this only affects the accuracy of the solution, and the projected gradient flow still performs correctly.

\subsection{Open-Loop Unstable System}
\label{subsec:ex2}
Consider the open-loop unstable system
\begin{align}
 A = \begin{bmatrix} 0 & 1 \\ -3 & 1\end{bmatrix},~ B_u = \begin{bmatrix} 0 \\ 1\end{bmatrix},~ B_d = \begin{bmatrix} -1 \\ 2\end{bmatrix},~ C = \begin{bmatrix} 1 & 1\end{bmatrix}, \label{math:sys_example2}
\end{align}
with $D=0$, $Q=I_2$, and $R=1$. Analogously to Section~\ref{subsec:ex1}, the stability and passivity regions are
\begin{align}
 \mathcal{K} &= \bigl\{ K\in\mathbb{R}^{1\times 2}\mid K_1 > -3,~ K_2 > 1\bigr\}, \notag\\
 \mathcal{P} &= \bigl\{ K\in\mathbb{R}^{1\times 2}\mid K_1 \geq -3,~ \tfrac{1}{2}(K_1+7)\leq K_2 \leq K_1+6\bigr\} \notag \\
 &\quad \setminus \bigl\{ \begin{bmatrix} -3 & 2\end{bmatrix}, \begin{bmatrix} -3 & 3\end{bmatrix}\bigr\}. \label{math:passivity_example2}
\end{align}
The optimal LQR gain $K^*\approx\begin{bmatrix} 0.162 & 2.525\end{bmatrix}$ violates $K_2 \geq \tfrac{1}{2}(K_1+7)$ and is thus not passivating. Fig.~\ref{fig:exam4} shows the verified cubes $\mathcal{C}_i$ with edge length $0.1$, the resulting convex inner approximation $\mathcal{P}_C$, and two trajectories of the projected gradient flow \eqref{math:proj_1}, which converge to $\hat K \approx \begin{bmatrix} 0.447 & 3.823\end{bmatrix}$.
  \begin{figure}[t]
    \centering
\includegraphics[width=\columnwidth]{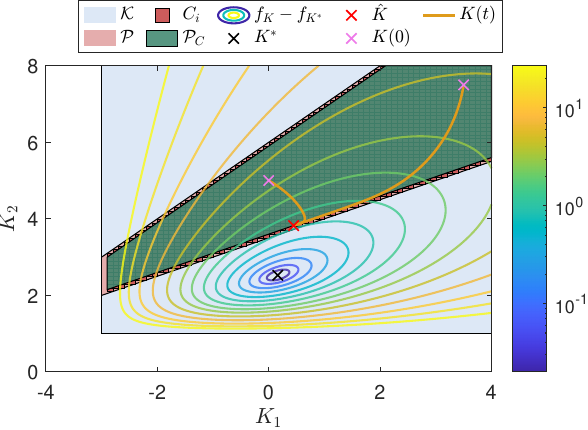}
 \caption{Visualization analogous to Fig.~\ref{fig:exam2} for the system \eqref{math:sys_example2} with cubes $\mathcal{C}_i$ having edge length $0.1$.}
 \label{fig:exam4}
\end{figure}

\subsection{Feedback Interconnection}
\label{subsec:ic}
To illustrate the motivation of Objective~\ref{obj}, the systems \eqref{math:sys_example} and \eqref{math:sys_example2}, indicated by the superscripts $(1)$ and $(2)$, are interconnected in negative feedback, i.e., $d^{(1)} = -y^{(2)}$ and $d^{(2)} = y^{(1)}$. With the decentralized gain $K=\operatorname{diag}(K^{(1)},K^{(2)})$, the closed loop is $\dot x = A_{\mathrm{cl}}x$ with $x = \begin{bmatrix} x^{(1)\top} & x^{(2)\top}\end{bmatrix}^\top$ and
\begin{align}
 A_{\mathrm{cl}} = \begin{bmatrix} A_K^{(1)} & -B_d^{(1)}C^{(2)} \\ B_d^{(2)} C^{(1)} & A_K^{(2)}\end{bmatrix}. \label{math:ic_closed}
\end{align}
Its LQR cost is $f_K = \operatorname{tr}(X_{\mathrm{cl}})$, where $X_{\mathrm{cl}}$ solves \eqref{math:Lyap1} with $A-B_uK$ replaced by $A_{\mathrm{cl}}$, $Q = I_4$, and $R = \operatorname{diag}(2,1)$.
Table~\ref{tab:ic} compares three gains: the optimal LQR gain $K^*_{\mathrm{c}}$ of the interconnected system, which is centralized and serves as a lower bound, the gain $K^*_{\mathrm{d}} = \operatorname{diag}(K^{*(1)},K^{*(2)})$ composed of the optimal LQR gains of the isolated systems, and the gain $\hat K_{\mathrm{d}} = \operatorname{diag}(\hat K^{(1)},\hat K^{(2)})$ composed of the gains $\hat K$ for the edge length $0.1$.
Strict passivity of both closed-loop systems under $\hat K^{(1)}$ and $\hat K^{(2)}$ is verified with the LMIs $L_{\text{s}}(0,K)$, so that the interconnection is asymptotically stable \cite[Thm.~6.3]{khalil_nonlinear_2002}, whereas the gains of $K^*_{\mathrm{d}}$ are not passivating.

\begin{table}[t]
\caption{Comparison of the gains for the negative feedback interconnection: LQR cost $f_K$ relative to $K^*_{\mathrm{c}}$ and worst-case ratio of $x_0^\top X_{\mathrm{cl}} x_0$ to that of $K^*_{\mathrm{c}}$ over all $x_0\neq 0$.}
\label{tab:ic}
\centering
\footnotesize
\setlength{\tabcolsep}{3pt}
\setlength{\arraycolsep}{2pt}
\begin{tabular}{ccccc}
\hline
Gain & $K^{(1)}$ & $K^{(2)}$ & $f_K$ & Ratio \\
\hline
$K^*_{\mathrm{c}}$ & \multicolumn{2}{c}{centralized} & $7.28$ ($100.0\,\%$) & $1.00$ \\
$K^*_{\mathrm{d}}$ & $\begin{bmatrix} 0.05 & 0.14 \end{bmatrix}$ & $\begin{bmatrix} 0.16 & 2.52 \end{bmatrix}$ & $11.72$ ($161.1\,\%$) & $2.01$ \\
$\hat K_{\mathrm{d}}$ & $\begin{bmatrix} -0.10 & 0.07 \end{bmatrix}$ & $\begin{bmatrix} 0.45 & 3.82 \end{bmatrix}$ & $13.45$ ($184.8\,\%$) & $2.32$ \\
\hline
\end{tabular}
\end{table}

Decentralization alone increases the cost by $61.1\,\%$, and requiring strict passivity of each closed-loop system adds another $15\,\%$. In return, the interconnection remains asymptotically stable if a system is replaced by another strictly passive system, without redesigning the remaining gains, whereas $K^*_{\mathrm{c}}$ has to be recomputed and $K^*_{\mathrm{d}}$ provides no such guarantee, although it stabilizes this particular interconnection.

\section{Conclusion and Outlook}

This work addressed the design of state-feedback gains that ensure (strict) passivity at an external port while achieving optimal LQR performance.
We established key topological properties of the passivating-gain set, including path-connectedness and contractibility. Building on this insight, we proposed an indirect approach based on inner-approximating the passivity region by a convex polytope and optimizing the LQR cost via a projected gradient flow. A remaining challenge is extending the convex inner approximation to higher-dimensional systems, as stated in Remark~\ref{remark:inner}. Addressing this limitation is an important direction for future research.

\bibliography{ifacconf.bib}       
\end{document}